\documentclass{article}
\usepackage{amsfonts}
\usepackage{amsmath}

\newtheorem{theorem}{Theorem}

\newtheorem{corollary}[theorem]{Corollary}

\newtheorem{definition}[theorem]{Definition}

\newtheorem{lemma}[theorem]{Lemma}

\newtheorem{remark}[theorem]{Remark}

\newenvironment{proof}[1][Proof]{\noindent\textbf{#1.} }{\ \rule{0.5em}{0.5em}}

\begin{document}

\begin{center}
\textbf{A transference result for Lebesgue spaces with A}$_{\infty }$\textbf{%
\ weights and its applications}

\textbf{Ramazan Akg\"{u}n}
\end{center}

\begin{quotation}
\textbf{Abstract} In this work we obtain a transference theorem for Lebesgue spaces with $A_{\infty }$ weights, namely, starting from some uniform-norm inequalities it is possible to obtain similar inequalities in Lebesgue
spaces with $A_{\infty }$ weights. This transference technic allows us to obtain some weighted norm inequalities easily. Also transference result gives possibility to use fractional difference operators in weighted Lebesgue spaces easier than the classical known one. We can obtain some norm-like inequalities easily as a consequence. Some important approximation inequalities of approximation by integral functions of finite degree can be obtained with a different proof.
\end{quotation}
\begin{quotation}
\end{quotation}
\begin{quotation}
\end{quotation}

\section{Introduction and main results}

\subsection{Preliminary Definitions}

We can give some preliminary definitions to state main results. A function $%
\omega :\mathbb{R}^{d}\mathbb{\rightarrow }\left[ 0,\infty \right] $ will be
called weight if $\omega $ is a measurable and positive function almost
everywhere (a.e.) on $\mathbb{R}^{d}.$ Define $\left\langle \omega
\right\rangle _{A}:=\int_{A}\omega (t)dt$ for $A\subset \mathbb{R}^{d}$. For
a weight $\omega $ on $\mathbb{R}^{d}$, we denote by $L_{p,\omega }$, $%
0<p\leq \infty $ the class of Lebesgue measurable functions $f:\mathbb{R}%
^{d}\rightarrow \mathbb{R}$ such that%
\[
\left\Vert f\right\Vert _{p,\omega }\equiv \left( \int\nolimits_{\mathbb{R}%
^{d}}\left\vert f\left( x\right) \right\vert ^{p}\omega \left( x\right)
dx\right) ^{1/p}<\infty \text{,\quad (}0<p<\infty \text{),}
\]%
\[
\left\Vert f\right\Vert _{\infty ,\omega }\equiv \left\Vert f\right\Vert
_{\infty }\equiv esssup_{x\in \mathbb{R}^{d}}\left\vert f\left(
x\right) \right\vert \text{,\quad }\left( p=\infty \right) \text{.}
\]%
We set $L_{p}\equiv L_{p,1}$ and $\left\Vert f\right\Vert _{p}\equiv
\left\Vert f\right\Vert _{p,1}$ and $\oint_{A}\omega (t)dt\equiv \left(
\left\vert A\right\vert ^{-1}\right) \int_{A}\omega (t)dt$ where $\left\vert
A\right\vert $ denotes the Lebesgue measure of a set $A\subset \mathbb{R}%
^{d}.$

For $1<p<\infty $ and $(1/p)+(1/p^{\prime })=1$ we set $\omega ^{\prime
}:=\omega ^{1-p^{\prime }}$ for a weight $\omega $. A weight $\omega $
satisfies the Muckenhoupt's condition $A_{p}$, $1\leq p<\infty $, if%
\begin{equation}
\left[ \omega \right] _{1}\equiv \sup\limits_{Q\in \mathbb{J}}(\left\vert
Q\right\vert ^{-1}\left\langle \omega \right\rangle _{Q}esssup_{x\in Q}(\omega \left( x\right) ^{-1}))<\infty 
\text{, \ }\left( p=1\right) \text{,}  \label{a1123}
\end{equation}%
\begin{equation}
\left[ \omega \right] _{p}\equiv \sup\limits_{Q\in \mathbb{J}}\left\vert
Q\right\vert ^{-p}\left\langle \omega \right\rangle _{Q}\left\langle \omega
^{\prime }\right\rangle _{Q}^{p-1}<\infty \text{,\quad }\left( 1<p<\infty
\right)  \label{ap}
\end{equation}%
with some finite constants independent of $Q$, where $\mathbb{J}$ is the
class of cubes in $\mathbb{R}^{d}$ with sides parallel to coordinate axes .

Define $A_{\infty }:=\cup _{1\leq p<\infty }A_{p}.$ It is well known that a
characterization of weights $\omega $ in the class $A_{\infty }$ is%
\[
\left[ \omega \right] _{\infty }\equiv \sup\nolimits_{Q}\int\nolimits_{Q}M%
\left[ \omega \left( y\right) \chi _{Q}\left( y\right) \right] dy<\infty
\]%
where $Q$ is any cube in $\mathbb{R}^{d}$ and $M$ is the Hardy-Littlewood
maximal function. Let $\mathbb{S}_{0}$ be the set of integrable simple
functions defined on $\mathbb{R}^{d}$.

\begin{definition}
\label{defReb} Let $u,x\in \mathbb{R}^{d}$, $\omega \in A_{\infty }$ and $%
f\in \mathbb{S}_{0}$. (i) Define weighted Steklov mean%
\[
S_{u,\omega }f\left( x\right) \equiv (\left\langle \omega \right\rangle _{ 
\left[ -1/2,1/2\right] ^{d}})^{-1}\int\nolimits_{\left[ -1/2,1/2\right]
^{d}}f\left( x+u+t\right) \omega \left( t\right) dt.
\]%
(ii) Define%
\[
\mathcal{R}_{u,\omega }f\left( x\right) \equiv \sum\limits_{k=0}^{1}\tfrac{1%
}{2^{k}}\tfrac{\left( S_{u,\omega }\right) ^{k}f\left( x\right) }{\left\Vert
S_{u,\omega }\right\Vert _{\mathcal{B}\left( L_{p,\omega },L_{p,\omega
}\right) }^{k}},\quad \left( S_{u,\omega }\right) ^{0}f\equiv f,
\]%
where $\left\Vert T\right\Vert _{\mathcal{B}\left( U,V\right) }$ is the
operator norm of a bounded operator $T:U\rightarrow V$.
\end{definition}

\begin{definition}
For given $\omega \in A_{\infty }$, $p\in \lbrack 1,\infty )$ we define the
class $\mathcal{Z}\left( p,\omega \right) \equiv $\{$g\in L_{p^{\prime
},\omega }$:$\left\Vert g\right\Vert _{p^{\prime },\omega }=1$\} where as
usual $p^{\prime }\equiv p/(p-1)$ for $p\in (1,\infty )$ and $1^{\prime
}\equiv \infty .$
\end{definition}

\begin{definition}
\label{ef} Let $\omega \in A_{\infty }$, $p\in \left( 0,\infty \right) $, $%
f\in L_{p,\omega }$. (a) For $p\in \lbrack 1,\infty )$ we define%
\begin{equation}
F_{f}\equiv F_{f}\left( u,G,p,\omega \right) \equiv \int\nolimits_{\mathbb{R}%
^{d}}\mathcal{R}_{u,\omega }f\left( x\right) \left\vert G(x)\right\vert
\omega \left( x\right) dx\text{,\quad }u\in \mathbb{R}^{d},  \label{efef}
\end{equation}%
with $G\in \mathcal{Z}\left( p,\omega \right) .$

(b) Let $p\in \left( 0,1\right) $. Since, there is $a_{0}\equiv e^{2^{11+d}%
\left[ \omega \right] _{\infty }}>1$ (see \cite[p.786]{hyPrz13}) such that,
we obtain $\omega \in A_{a}$ with $a\equiv a_{0}+0,01$. Then, one can get a $%
q\in \left( 0,p\right) $ such that $\omega \in A_{p/q}$ (Take for example
any $q$ less than $p/a_{0}$). Now, set $r\equiv p/q$ and define%
\begin{equation}
F_{f}\equiv F_{f}\left( u,G,r,\omega \right) \equiv \int\nolimits_{\mathbb{R}%
^{d}}\left( \mathcal{R}_{u,\omega }f\left( x\right) \right) ^{q}\left\vert
G(x)\right\vert \omega \left( x\right) dx\text{,\quad }u\in \mathbb{R}^{d}
\label{fefe}
\end{equation}%
with $G\in \mathcal{Z}\left( r,\omega \right) .$
\end{definition}

Let $\mathcal{C}(\mathbb{R}^{d})$ be the class of bounded, uniformly
continuous functions defined on $\mathbb{R}^{d}$ and $\left\Vert
f\right\Vert _{\mathcal{C}(\mathbb{R}^{d})}:=\sup \left\{ \left\vert f\left(
t\right) \right\vert :t\in \mathbb{R}^{d}\right\} $ for $f\in \mathcal{C}(%
\mathbb{R}^{d})$.

\begin{remark}
Note that, by Theorem \ref{UCB}, $F_{f}\in \mathcal{C}(\mathbb{R}^{d})$ for $%
\omega \in A_{\infty }$, $p\in \left( 0,\infty \right) $, and $f\in
L_{p,\omega }.$
\end{remark}

\subsection{Main results}

To obtain a weighted norm inequality of the following type%
\begin{equation}
\left\Vert f\right\Vert _{p,\omega }\leq c\left\Vert g\right\Vert _{p,\omega
}  \label{WNI}
\end{equation}%
for $0<p<\infty $, $\omega \in A_{\infty }$, $f\in L_{p,\omega }$, we define
an intermediate function as in Definition \ref{ef}%
\[
F_{f}:\mathbb{R}^{d}\mathbb{\rightarrow }\mathcal{C}\left( \mathbb{R}%
^{d}\right) \text{,\quad }u\mapsto F_{f}\left( u\right)
\]%
having properties%
\[
\left\Vert f\right\Vert _{p,\omega }\leq c\left\Vert F_{f}\left( \cdot
\right) \right\Vert _{\mathcal{C}(\mathbb{R}^{d})}\text{ and }\left\Vert
F_{g}\left( \cdot \right) \right\Vert _{\mathcal{C}(\mathbb{R}^{d})}\leq
c\left\Vert g\right\Vert _{p,\omega }
\]%
for some positive constants. Now, if the following uniform norm estimate%
\[
\left\Vert F_{f}\left( \cdot \right) \right\Vert _{\mathcal{C}\left( \mathbb{%
R}^{d}\right) }\leq c\left\Vert F_{g}\left( \cdot \right) \right\Vert _{%
\mathcal{C}\left( \mathbb{R}^{d}\right) }
\]%
holds, then, we obtain desired weighted norm inequality (\ref{WNI}).

All constants $c>0$ will be some positive number such that, they depend on
the main parameters in question, and change in each occurrences.

Main theorem of this work is the following transference result.

\begin{theorem}
\label{TR} Let $\mathcal{F}$ be a family of couples $\left( f,g\right) $ of
nonnegative functions, $d\in \mathbb{N}$, $p\in \left( 0,\infty \right) $, $%
\omega \in A_{\infty }$ and $f\in L_{p,\omega }$. Suppose that the following
uniform-norm estimate holds%
\begin{equation}
\left\Vert F_{f}\left( \cdot ,G,p,\omega \right) \right\Vert _{\mathcal{C}%
\left( \mathbb{R}^{d}\right) }\leq c\left\Vert F_{g}\left( \left( \cdot
,G^{\ast },p,\omega \right) \right) \right\Vert _{\mathcal{C}\left( \mathbb{R%
}^{d}\right) }  \label{unif}
\end{equation}%
for any $G,G^{\ast }\in \mathcal{Z}\left( p/q,\omega \right) $ where $q:=1$
for $p\in \lbrack 1,\infty )$ and $q\in \left( 0,p\right) $ for $p\in \left(
0,1\right) $. Then, weighted norm inequality%
\[
\left\Vert f\right\Vert _{p,\omega }\leq c\left\Vert g\right\Vert _{p,\omega
}\text{,\quad }\left( f,g\right) \in \mathcal{F}\text{,}
\]%
holds with a positive constant $c:=c\left( q,\omega ,p,d\right) $.
\end{theorem}

As a corollary of Theorem \ref{TR} we can easily obtain norm inequalities
with a suitable fractional difference operator $\left( E-V_{\delta }\right)
^{r}$ with $\delta >0$ in weighted Lebesgue spaces $L_{p,\omega }\equiv
L^{p}(\omega \left( x\right) dx)$\textbf{\ }where $E$ is the identity
operator and $V_{\delta }$ is a suitable translation operator in $%
L_{p,\omega }$ with $\omega \in A_{\infty }$.

Theorem \ref{TR} gives several important norm-like inequalities. For
example, one can consider reverse sharp Marchaud inequality: If $p\in \left(
1,\infty \right) $, $\omega \in A_{\infty }$, $f\in L_{p,\omega }$, then,
there are $m\in \mathbb{N}$ and $a>1$ such that%
\begin{equation}
\left\Vert \left( E-V_{\delta }\right) ^{r}f\right\Vert _{p,\omega }^{s}\geq
c\sum\limits_{j=0}^{m}2^{-j2rs}\left\Vert \left( E-V_{2^{j}\delta }\right)
^{r+1}f\right\Vert _{p,\omega }^{s},  \label{revShM}
\end{equation}%
holds for $s:=\max \left\{ 2,a\right\} $ with a positive constant $%
c:=c\left( s,\omega ,p,d\right) $ depending only on $s,\omega ,p,d.$

To obtain inequalities of type (\ref{revShM}) with the classical
extrapolation theorem seem to be not possible. Note also that, classical
direct proof of (\ref{revShM})-type inequality (\cite[Theorem 2.1]{Di-Pr12})
is also hard to overcome for weighted spaces.

To give proof of inequality (\ref{revShM}) we need to obtain a version of
main Theorem \ref{TR}.

\begin{definition}
Let $B$ be a Banach space with norm $\left\Vert \cdot \right\Vert _{B}.$ For
an $m\in \mathbb{N}$ and $s\in \left( 0,\infty \right) $, we define%
\[
\left\Vert f\right\Vert _{l_{s}^{m}\left( B\right) }\equiv \left\Vert \left(
f_{j}\right) _{j=0}^{m}\right\Vert _{l_{s}^{m}\left( B\right) }\equiv \left(
\sum\nolimits_{j=0}^{m}\left\Vert f_{j}\right\Vert _{B}^{s}\right) ^{1/s}.
\]
\end{definition}

\begin{theorem}
\label{TRver} Let $\mathcal{F}$ be a family of couples $\left( f,g\right) $
of nonnegative functions, $d\in \mathbb{N}$, $p\in \left( 1,\infty \right) $
and $\omega \in A_{\infty }$. Suppose that there exists an $a\in \left(
1,\infty \right) $ such that, for any $G,G^{\ast }\in \mathcal{Z}\left(
p/q,\omega \right) ,$%
\begin{equation}
\left\Vert F_{f}\left( \cdot ,G,p,\omega \right) \right\Vert _{L_{a}}\leq
c\left\Vert F_{g_{j}}\left( \cdot ,G^{\ast },p,\omega \right) \right\Vert
_{l_{s}^{m}\left( L_{a}\right) }\text{,\quad }\left( f,g_{j}\right) \in 
\mathcal{F}\text{,}  \label{hptz}
\end{equation}
holds for some $s>2$, and positive $c:=c\left( m,a,s,d\right) $ provided the
left hand side (\ref{hptz}) is finite, where $q:=1$ for $p\in \lbrack
1,\infty )$ and $q\in \left( 0,p\right) $ for $p\in \left( 0,1\right) $. Then%
\begin{equation}
\left\Vert f\right\Vert _{p,\omega }\leq c\left\Vert g_{j}\right\Vert
_{l_{s}^{m}\left( L_{p,\omega }\right) }\text{,\quad }\left( f,g_{j}\right)
\in \mathcal{F}\text{,}  \label{hkm}
\end{equation}%
holds with a positive $c:=c\left( m,a,s,p,\omega ,d\right) $ when the
left-hand side (\ref{hkm}) is finite.
\end{theorem}

Note that, other versions of main Theorem \ref{TR} is also possible for
other versions of weighted norm inequalities.

On the other hand, by using Theorem \ref{TR}, many of the basic inequalities
of approximation by entire functions of finite degree can be obtained by
extrapolation argument as an alternative proof. See proof of Theorem \ref%
{Jacks}.

To prove main properties of intermediate functions $F_{f}$ we need some
preliminary observations related to (weighted) Steklov averages.

\begin{definition}
Define Steklov mean, for $u,x\in \mathbb{R}^{d}$, $1\leq p<\infty $, $\omega
\in A_{p}$ and $f\in L_{p,\omega }$, as%
\[
S_{u}f\left( x\right) \equiv \int_{\left[ -1/2,1/2\right] ^{d}}f\left(
x+u+t\right) dt.
\]
\end{definition}

\begin{theorem}
\label{Suf} We suppose that $1\leq p<\infty $, $\omega \in A_{p}$ and $f\in
L_{p,\omega }$. In this case, for any $u\in \mathbb{R}^{d}$, there holds%
\[
\left\Vert S_{u}f\right\Vert _{p,\omega }\leq 3^{2d+1/p}\left[ \omega \right]
_{p}^{1/p}\left\Vert f\right\Vert _{p,\omega }.
\]
\end{theorem}

\begin{remark}
\label{remCc} (a) By theorem 18.3 of \cite{yeh}, the class $\mathbb{S}_{0}$
of integrable simple functions defined on $\mathbb{R}^{d}$, is a dense
subset of $L_{p,\omega }$ with $0<p<\infty $, and $\omega \in A_{\infty }$.

(b) By Theorems 18.3, 19.37 and Observation 7.5 of \cite{yeh}, we can
observe that the class $C_{c}\equiv C_{c}\left( \mathbb{R}^{d}\right) $ of
continuous functions of compact support, is a dense subset of $L_{p,\omega }$
with $0<p<\infty $, and $\omega \in A_{\infty }$. Note that, Theorem 18.3 of 
\cite{yeh} is proved for $1\leq p<\infty $ but the same proof holds also for 
$0<p<1$.
\end{remark}

Now, using Theorem \ref{Suf} we can prove the following result.

\begin{theorem}
\label{Suwf} We suppose that $0<p<\infty $ and $\omega \in A_{\infty }.$ In
this case, for any $u\in \mathbb{R}^{d}$, and $f\in \mathbb{S}_{0}$, there
holds%
\begin{equation}
\left\Vert S_{u,\omega }f\right\Vert _{p,\omega }\leq c\left\Vert
f\right\Vert _{p,\omega }  \label{Insuwf}
\end{equation}%
with a positive constant $c=c(d,p,\omega )$.
\end{theorem}

\begin{remark}
It is clear from its definition that $\mathcal{R}_{u,\omega }\left\vert
f\right\vert \geq \left\vert f\right\vert .$
\end{remark}

Now, using Theorem \ref{Suwf} we can prove the following result.

\begin{theorem}
\label{Ruwf} We suppose that $0<p<\infty $ and $\omega \in A_{\infty }.$ In
this case, for any $u\in \mathbb{R}^{d}$, and $f\in \mathbb{S}_{0}$, there
holds%
\begin{equation}
\left\Vert \mathcal{R}_{u,\omega }f\right\Vert _{p,\omega }\leq 4^{1/\min
\left\{ 1,p\right\} }\left\Vert f\right\Vert _{p,\omega }.  \label{InRuwf}
\end{equation}
\end{theorem}

\begin{theorem}
\label{UCB} Let $0<p<\infty $, $\omega \in A_{\infty }$, and $f\in
L_{p,\omega }$. In this case, the function $F_{f}$ defined in (\ref{efef})
or (\ref{fefe}) is bounded, uniformly continuous function on $\mathbb{R}^{d}$%
.
\end{theorem}

\section{Applications on operators}

We can give several corollaries that can be obtained easily by using Theorem %
\ref{TR}.

We define the following operators.

\begin{definition}
For $\delta >0$, $x,u\in \mathbb{R}^{d}$ and locally integrable functions $f:%
\mathbb{R}^{d}\rightarrow \mathbb{R}$, we define operators%
\begin{equation}
S_{\delta ,u}f\left( x\right) \equiv \int\nolimits_{\left[ -\delta
/2,\delta /2\right] ^{d}}f\left( x+u+s\right) ds,  \label{TrStekl}
\end{equation}%
\begin{equation}
V_{\delta }f\left( x\right) \equiv S_{\delta ,0}f\left( x\right) \equiv
\int\nolimits_{\left[ -\delta /2,\delta /2\right] ^{d}}f\left( x+s\right)
ds,  \label{Ved}
\end{equation}%
\begin{equation}
Z_{\delta }f\left( x\right) \equiv \int\nolimits_{\left[ \delta /2,\delta %
\right] ^{d}}f\left( x+t\right) dt,  \label{Zed}
\end{equation}%
\begin{equation}
\mathcal{B}_{\delta }f\left( x\right) \equiv \int\nolimits_{\left[
0,\delta \right] ^{d}}f\left( x+v\right) dv.  \label{Bed}
\end{equation}
\end{definition}

\begin{theorem}
\label{Sduf} We suppose that $1<p<\infty $, $\omega \in A_{\infty }$, $%
u,v\in \mathbb{R}^{d}$, and $\delta \in \left( 0,\infty \right) $. Let an
operator $\Gamma $ represents operator $f\rightarrow S_{\delta ,v}f$. In
this case, for $f\in L_{p,\omega }$, there hold properties%
\begin{equation}
S_{u,\omega }\left( \Gamma f\right) =\Gamma \left( S_{u,\omega }f\right)
,\qquad \mathcal{R}_{u,\omega }\left( \Gamma f\right) =\Gamma \left( 
\mathcal{R}_{u,\omega }f\right) ,  \label{Sduw01}
\end{equation}%
\begin{equation}
F_{\Gamma f}\text{=}\Gamma \left( F_{f}\right) \text{,\qquad }\left\Vert
\Gamma f\right\Vert _{p,\omega }\leq c\left\Vert f\right\Vert _{p,\omega },
\label{InSduf}
\end{equation}%
with a positive constant $c=c(d,p,\omega )$.
\end{theorem}

As a corollary of Theorem \ref{Sduf} and Theorem \ref{ET} we have the
following result.

\begin{corollary}
\label{CorSduf} We suppose that $0<p<\infty $, $\omega \in A_{\infty }$ and $%
v\in \mathbb{R}^{d}$, $\delta \in \left( 0,\infty \right) $. Let an operator 
$\Gamma $ represents operator $f\rightarrow S_{\delta ,v}f$. In this case,
for any $f\in L_{p,\omega }$ ($f\in L_{p,\omega }\cap \mathbb{S}_{0}$ when $%
0<p<1$), there holds%
\[
\left\Vert \Gamma f\right\Vert _{p,\omega }\leq c\left\Vert f\right\Vert
_{p,\omega },
\]%
with a positive constant $c=c(d,p,\omega )$.
\end{corollary}

Proof of the following two results for several operators are the same with
Theorem \ref{Sduf} and Corollary \ref{CorSduf}.

\begin{theorem}
\label{oper1} We suppose that $1<p<\infty $, $\omega \in A_{\infty }$, $%
u,v\in \mathbb{R}^{d}$, and $\delta \in \left( 0,\infty \right) $. If we
replace the operator $\Gamma $ in Theorem \ref{Sduf}, by one of the
following operators $f\rightarrow S_{v}f$, $f\rightarrow S_{v,\omega }f$, $%
f\rightarrow \mathcal{R}_{v,\omega }f$, $f\rightarrow V_{v}f$, $f\rightarrow
Z_{v}f$, or $f\rightarrow \mathcal{B}_{v}f$ then, conclusion of Theorem \ref%
{Sduf} is remain valid.
\end{theorem}

\begin{corollary}
\label{Coroper1} We suppose that $0<p<\infty $, $\omega \in A_{\infty }$, $%
u,v\in \mathbb{R}^{d}$, and $\delta \in \left( 0,\infty \right) $. If we
replace the operator $\Gamma $ in Corollary \ref{CorSduf}, by one of the
following operators $f\rightarrow S_{v}f$, $f\rightarrow S_{v,\omega }f$, $%
f\rightarrow \mathcal{R}_{v,\omega }f$, $f\rightarrow V_{v}f$, $f\rightarrow
Z_{v}f$, or $f\rightarrow \mathcal{B}_{v}f$, then, conclusion of Corollary %
\ref{CorSduf} is remain valid.
\end{corollary}

Fractional Difference Operator can be defined as follows.

\begin{definition}
Let $\delta ,k\in \left( 0,\infty \right) $ and define difference $\left(
E-V_{\delta }\right) ^{k}$ of fractional order $k$ at $x\in \mathbb{R}^{d}$
with step $\delta ,$ by%
\begin{equation}
\left( E-V_{\delta }\right) ^{k}f\left( x\right) \equiv
\sum\nolimits_{s=0}^{\infty }(-1)^{s}C_{s}^{k}\left( V_{\delta }\right)
^{s}f\left( x\right)  \label{SerFra}
\end{equation}%
where $C_{0}^{k}\equiv 1$, and $C_{s}^{k}\equiv \Pi _{n=1}^{s}\tfrac{k-n+1}{n%
}$ are binomial coefficients.
\end{definition}

\begin{theorem}
\label{frac} We suppose that $0<p<\infty $ and $\omega \in A_{\infty }.$ In
this case, for any $\delta \in \left( 0,\infty \right) $, and $f\in \mathbb{S%
}_{0}$, there holds%
\begin{equation}
\left\Vert \left( E-V_{\delta }\right) ^{k}f\right\Vert _{p,\omega }\leq
c\left\Vert f\right\Vert _{p,\omega },  \label{InFrac}
\end{equation}%
with a positive constant $c:=c(s,\omega ,p,d)$ depending only on $s,\omega
,p.$
\end{theorem}

Theorem \ref{TR} gives several important norm-like inequalities. For
example, one can consider weighted reverse sharp Marchaud inequality:

\begin{theorem}
\label{ShrpMarch} If $r\in \mathbb{N}$, $\omega \in A_{\infty }$, $p\in
\left( 1,\infty \right) $, $\delta \in \left( 0,\infty \right) $, and $f\in
L_{p,\omega }$, then, there are $m\in \mathbb{N}$ and $a>1$ such that%
\begin{equation}
\left\Vert \left( E\text{-}V_{\delta }\right) ^{r}f\right\Vert _{p,\omega
}^{s}\geq c\sum\limits_{j=0}^{m}2^{-j2rs}\left\Vert \left( E\text{-}%
V_{2^{j}\delta }\right) ^{r+1}f\right\Vert _{p,\omega }^{s},
\label{InShrpMarch}
\end{equation}%
holds for $s:=\max \left\{ 2,a\right\} $ with a constant $c>0$ depending
only on $s,\omega ,p,d.$
\end{theorem}

\section{Applications on approximation by exponential type functions}

\begin{definition}
Let $X:=L^{p}\left( \mathbb{R}^{d}\right) $ or $L_{p,\omega }$ or $\mathcal{C%
}\left( \mathbb{R}^{d}\right) $.

(i) We define $\mathcal{G}_{\sigma }\left( X\right) $ as the class of entire
function of exponential type $\sigma >0$ that belongs to $X$, namely, "$g\in 
\mathcal{G}_{\sigma }\left( X\right) $ iff supp$\hat{g}\left( \mathbf{y}%
\right) \subset \left\{ \mathbf{y}:\left\vert \mathbf{y}\right\vert \leq
\sigma \right\} $ and $g\in X$" where $\hat{g}$ is the Fourier transform of $%
g$.

We set $\mathcal{G}_{\sigma }\left( p\right) :=\mathcal{G}_{\sigma }\left(
L^{p}\left( \mathbb{R}^{d}\right) \right) $, $\mathcal{G}_{\sigma }\left( p%
\mathbf{,}\omega \right) :=\mathcal{G}_{\sigma }\left( L_{p,\omega }\right) $%
, and $\mathcal{G}_{\sigma }\left( \mathcal{C}\right) :=\mathcal{G}_{\sigma
}\left( \mathcal{C}\left( \mathbb{R}^{d}\right) \right) $.

(ii) The best approximation in $L_{p,\omega }$ by functions of exponential
type is given by%
\begin{equation}
A_{\sigma }(f)_{X}:=\inf\nolimits_{g}\{\Vert f-g\Vert _{X}:g\in \mathcal{G}%
_{\sigma }\left( X\right) \}.  \label{fd}
\end{equation}%
Let $A_{\sigma }(f)_{p}$:=$A_{\sigma }(f)_{L^{p}\left( \mathbb{R}^{d}\right)
}$, $A_{\sigma }(f)_{p\mathbf{,}\omega }$:=$A_{\sigma }(f)_{L_{p,\omega }}$,
and $A_{\sigma }(f)_{\mathcal{C}}$:=$A_{\sigma }(f)_{\mathcal{C}\left( 
\mathbb{R}^{d}\right) }.$
\end{definition}

\begin{definition}
Let $\sigma >0$, $1\leq p\leq \infty $, $f\in L^{p}\left( \mathbb{R}%
^{d}\right) $,%
\[
\vartheta _{\sigma }\left( t\right) :=\frac{1}{\sigma ^{d}}%
\prod\limits_{j=1}^{d}\frac{\cos \left( \sigma t_{j}\right) -\cos \left(
2\sigma t_{j}\right) }{t_{j}^{2}}\text{, \ \ }t\in \mathbb{R}^{d},
\]%
and%
\[
J\left( f,\sigma \right) \left( x\right) =\frac{1}{\pi ^{d}}\int\nolimits_{%
\mathbb{R}^{d}}\vartheta _{\sigma }\left( x-u\right) f\left( u\right) du%
\text{, \ \ }x\in \mathbb{R}^{d},
\]%
be the de l\`{a} Val\`{e}e Poussin operator (\cite[pp. 304-306; (11)]%
{SMNbook}).
\end{definition}

\begin{theorem}
\label{dela}(\cite[pp. 304-306]{SMNbook}) It is known that, if $f\in
L^{p}\left( \mathbb{R}^{d}\right) $, $1\leq p\leq \infty $, then,

(i) $J\left( f,\sigma \right) \in \mathcal{G}_{2\sigma }\left( p\right) $,

(ii) $J\left( g_{\sigma },\sigma \right) =g_{\sigma }$ for any $g_{\sigma
}\in \mathcal{G}_{\sigma }\left( p\right) $,

(iii) $\Vert J\left( f,\sigma \right) \Vert _{L_{p}\left( \mathbb{R}%
^{d}\right) }\leq c\Vert f\Vert _{L_{p}\left( \mathbb{R}^{d}\right) }$.
\end{theorem}

\begin{theorem}
\label{Jacks} We suppose that $0<p<\infty $ and $\omega \in A_{\infty }.$ In
this case, for any $\sigma \in \left( 0,\infty \right) $, and $f\in
L_{p,\omega }$ ($f\in L_{p,\omega }\cap \mathbb{S}_{0}$ when $0<p<1$), there
holds%
\begin{equation}
A_{\sigma }\left( f\right) _{p,\omega }\leq c\left\Vert \left( I-V_{1/\sigma
}\right) ^{r}f\right\Vert _{p,\omega }  \label{InJT}
\end{equation}%
with some constant $c>0$ depending $p,\omega ,d$ only.
\end{theorem}

After the results of S. N. Bernstein \cite[1912]{B12}, some systematic
studies on approximation by exponential functions of degree$\leq \sigma $
for $d=1$ or $d>1$, continued by A. F. Timan \cite{AFT}, N. I. Akhieser \cite%
{Ack}, S. M. Nikolski \cite{SMNbook}, I. I. Ibragimov \cite{II3}, H. Triebel 
\cite{HT83}, P. L. Butzer, H. J. Schmeisser and W. Sickel \cite{hjssic}, R.
M. Trigub and E. S. Belinsky \cite{TB04}. These reference books contain
several inequalities of exponential functions of degree$\leq \sigma $ in
spaces $L^{p}(\mathbb{R}^{d})$ with $1\leq p\leq \infty $. Some other works
also include results of approximation by exponential functions of degree$%
\leq \sigma .$ See for example, \cite{AG}, \cite{art}, \cite{DDTi}, \cite%
{Dit}, \cite{Di-Pr12}, \cite{DiRu}, \cite{gaim1}, \cite{GK}, \cite{KolTik}, 
\cite{FGN}, \cite{LD}, \cite{ponom}, \cite{Ste1}, \cite{Po}, \cite{T81}, 
\cite{MFT61}, \cite{vak1}. For periodic $\omega \mathbf{\in }A_{p}$, $%
1<p<\infty $ and periodic $f\in L_{p,\omega }$, ($d=1$) some results on
trigonometric approximation are known. See e.g. \cite{ascs}, \cite{AK1}, 
\cite{spbu}, \cite{ahak}, \cite{Ky2}, \cite{Ja1}, \cite{ja2}, \cite{yeydmi11}%
.

\section{Proofs}

Suppose that $Q\left( x,\varepsilon \right) $ denotes the cube with center $%
x $ and sidelenght $2\varepsilon .$

\begin{definition}
(\cite[Def. 4.4.2; pp: 115-116]{dhhr11}) (a) A family $\Psi $ of measurable
sets $U\subset \mathbb{R}^{d}$ is called locally $N$-finite ($N\in \mathbb{N}
$) if 
\[
\sum\nolimits_{U\in \Psi }\chi _{U}\left( x\right) \leq N
\]%
almost everywhere in $\mathbb{R}^{d}$ where $\chi _{U}$ is the
characteristic function of the set $U$.

(b) A family $\Psi $ of open bounded sets $U\subset \mathbb{R}^{d}$ is
locally $1$-finite if and only if the sets $U\in \Psi $ are pairwise
disjoint.
\end{definition}

\begin{definition}
Suppose that $B$ is a Banach space on $\mathbb{R}^{d}$ with norm $\left\Vert
\cdot \right\Vert _{B}$. We set, for $f\in B$, $r\in \mathbb{N}$ and $\delta
>0$,%
\[
\inf_{\Delta ^{r}g\in B}\left\{ \left\Vert f-g\right\Vert _{B}+\delta
^{r}\left\Vert \Delta ^{r}g\right\Vert _{B}\right\} \equiv K_{\Delta
^{r}}\left( f,\delta ^{r},B\right) ,
\]%
where $\Delta f\equiv f_{x_{1}x_{1}}+...+f_{x_{d}x_{d}}$ is Laplace
transform and $\Delta ^{r}$ is $r$th iterate of $\Delta $.
\end{definition}

\begin{lemma}
(\cite[Theorem 16.14]{yeh}) Let $1<p<\infty $, $\omega $ be a weight, $f\in
L_{p,\omega }$ and $g\in L_{p^{\prime },\omega }$. In this case, H\"{o}%
lder's inequality%
\begin{equation}
\int\nolimits_{\mathbb{R}^{d}}\left\vert f(x)g(x)\right\vert \omega \left(
x\right) dx\leq \left\Vert f\right\Vert _{p,\omega }\left\Vert g\right\Vert
_{p^{\prime },\omega }  \label{holder}
\end{equation}%
holds.
\end{lemma}

\begin{theorem}
\label{ET}(\cite{CUMP}) Let $\mathcal{F}$ be a family of couples $\left(
f,g\right) $ of nonnegative functions and $d\in \mathbb{N}$. Suppose that,
for some $p_{0}\in \left( 0,\infty \right) $ and for every weight $\omega
\in A_{\infty }$ there holds inequality%
\begin{equation}
\int\nolimits_{\mathbb{R}^{d}}f\left( x\right) ^{p_{0}}\omega \left(
x\right) dx\leq c\int\nolimits_{\mathbb{R}^{d}}g\left( x\right)
^{p_{0}}\omega \left( x\right) dx\text{,\quad }\left( f,g\right) \in 
\mathcal{F}\text{,}  \label{weHypth}
\end{equation}%
provided the left hand side is finite. Then, for all $p\in \left( 0,\infty
\right) $ and all $\omega \in A_{\infty }$,%
\begin{equation}
\int\nolimits_{\mathbb{R}^{d}}f\left( x\right) ^{p}\omega \left( x\right)
dx\leq c\int\nolimits_{\mathbb{R}^{d}}g\left( x\right) ^{p}\omega \left(
x\right) dx\text{,\quad }\left( f,g\right) \in \mathcal{F}\text{,}
\label{concl}
\end{equation}%
holds when the left-hand side is finite.
\end{theorem}

\begin{proof}[\textbf{Proof of Theorem \protect\ref{Suf}}]
Let $\Psi $ be $1$-finite family of open bounded cubes $Q_{i}$ of $\mathbb{R}%
^{d}$ having Lebesgue measure $1$ and with sides parallel to coordinate
axes, such that $\left( \cup _{i}Q_{i}\right) \cup A=\mathbb{R}^{d}$ for
some null-set $A$. Since $u\in \mathbb{R}^{d}$ there exists $m\in \mathbb{Z}%
^{d}$ such that $m\leq u<(m+2)$. Let $Q+m$ be translation of the cube $Q$ by
vector $m$. We set $\left( Q_{i}+m\right) ^{\pm }:=\left( Q_{i-1}\cup
Q_{i}\cup Q_{i+1}\right) +m$. Then%
\[
\left\Vert S_{u}f\right\Vert _{p,\omega }^{p}=\sum\limits_{Q_{i}\in \Psi
}\int\limits_{Q_{i}}\left\vert \oint\limits_{\left[ -1/2,1/2\right]
^{d}}f(x+u+t)dt\right\vert ^{p}\omega (x)dx
\]%
\[
\leq \sum\limits_{Q_{i}\in \Psi }\int\limits_{Q_{i}}\left[
\oint\limits_{Q\left( x+u,1/2\right) }\omega ^{\frac{1}{p}}(t)\left\vert
f(t)\right\vert \omega ^{\frac{-1}{p}}(t)dt\right] ^{p}\omega (x)dx
\]%
\[
\leq \sum\limits_{Q_{i}\in \Psi }\int\limits_{Q_{i}}\left[ \left(
\oint\limits_{Q\left( x+u,1/2\right) }\omega (t)\left\vert f(t)\right\vert
^{p}dt\right) ^{\frac{1}{p}}\left( \oint\limits_{Q\left( x+u,1/2\right)
}\omega ^{\frac{-p^{\prime }}{p}}(t)dt\right) ^{\frac{1}{p^{\prime }}}\right]
^{p}\omega (x)dx
\]%
\[
\leq \sum\limits_{Q_{i}\in \Psi }\int\limits_{Q_{i}}\int\limits_{Q\left(
x+u,1/2\right) }\omega (t)\left\vert f(t)\right\vert ^{p}dt\left(
\oint\limits_{Q\left( x+u,1/2\right) }\omega ^{\frac{-p^{\prime }}{p}%
}(t)dt\right) ^{\frac{p}{p^{\prime }}}\omega (x)dx
\]%
\[
\leq 3^{2dp}\sum\limits_{Q_{i}\in \Psi }\oint\limits_{\left( Q_{i}+m\right)
^{\pm }}\omega (x)dx\left( \oint\limits_{\left( Q_{i}+m\right) ^{\pm
}}\omega ^{\frac{-1}{p-1}}(t)dt\right) ^{p-1}\int\limits_{\left(
Q_{i}+m\right) ^{\pm }}\omega (t)\left\vert f(t)\right\vert ^{p}dt
\]%
\[
\leq 3^{2dp}\left[ \omega \right] _{p}\sum\limits_{Q_{i}\in \Psi
}\int\limits_{\left( Q_{i}+m\right) ^{\pm }}\left\vert f(t)\right\vert
^{p}\omega (t)dt
\]%
\[
\leq 3^{2dp}\left[ \omega \right] _{p}\sum\limits_{Q_{i}\in \Psi }\left\{
\int\limits_{Q_{i-1}+m}+\int\limits_{Q_{i}+m}+\int\limits_{Q_{i+1}+m}\right%
\} \left\vert f(t)\right\vert ^{p}\omega (t)dt
\]%
\[
\leq 3^{2dp}\left[ \omega \right] _{p}\int\limits_{\mathbb{R}^{d}}\left\vert
f(t)\right\vert ^{p}\omega (t)\left\{ \sum\limits_{Q_{i}\in \Psi }\left(
\chi _{Q_{i-1}+m\left( t\right) }+\chi _{Q_{i}+m\left( t\right) }+\chi
_{Q_{i+1}+m\left( t\right) }\right) \right\} dt
\]%
\[
\leq 3^{2dp+1}\left[ \omega \right] _{p}\left\Vert f\right\Vert _{p,\omega
}^{p}.
\]

For $p=1$ we find%
\[
\left\Vert S_{u}f\right\Vert _{1,\omega }=\sum\limits_{Q_{i}\in \Psi
}\int\limits_{Q_{i}}\left\vert \oint\limits_{\left[ -1/2,1/2\right]
^{d}}f(x+u+t)dt\right\vert \omega (x)dx
\]%
\[
\leq \sum\limits_{Q_{i}\in \Psi }\int\limits_{Q_{i}}\int\limits_{Q\left(
x+u,1/2\right) }\omega (t)\left\vert f(t)\right\vert \frac{1}{\omega (t)}%
dt\omega (x)dx
\]%
\[
\leq 3^{d}\sum\limits_{Q_{i}\in \Psi }\frac{1}{\left\vert \left(
Q_{i}+m\right) ^{\pm }\right\vert }\int\limits_{\left( Q_{i}+m\right) ^{\pm
}}\omega (x)dx\left( \underset{t\in \left( Q_{i}+m\right) ^{\pm }}{esssup}\frac{1}{\omega (t)}\right) \int\limits_{\left( Q_{i}+m\right) ^{\pm
}}\left\vert f(t)\right\vert \omega (t)dt
\]%
\[
\leq 3^{d}\left[ \gamma \right] _{1}\sum\limits_{Q_{i}\in \Psi }\left\{
\int\limits_{Q_{i-1}+m}+\int\limits_{Q_{i}+m}+\int\limits_{Q_{i+1}+m}\right%
\} \left\vert f(t)\right\vert \omega (t)dt
\]%
\[
=3^{d}\left[ \gamma \right] _{1}\int\limits_{\mathbb{R}^{d}}\left\vert
f(t)\right\vert \omega (t)\left\{ \sum\limits_{Q_{i}\in Q}\chi
_{Q_{i-1}+m}\left( t\right) +\sum\limits_{Q_{i}\in Q}\chi _{Q_{i}+m}\left(
t\right) +\sum\limits_{Q_{i}\in Q}\chi _{Q_{i+1}+m}\left( t\right) \right\}
dt
\]%
\[
\leq 3^{d+1}\left[ \gamma \right] _{1}\int\limits_{\mathbb{R}^{d}}\left\vert
f(t)\right\vert \omega (t)dt=3^{d+1}\left[ \gamma \right] _{1}\left\Vert
f\right\Vert _{1,\omega },
\]%
as required.
\end{proof}

\begin{proof}[\textbf{Proof of Theorem \protect\ref{Suwf}}]
Let $\omega \in A_{\infty }$. (a) First, we consider the case $p\in \left(
1,\infty \right) $. Suppose that $f\in L_{p,\omega }$. Then, there is $%
a_{0}\equiv e^{2^{11+d}\left[ \omega \right] _{\infty }}>1$ (see \cite[p.786]%
{hyPrz13}) such that, for $\tilde{p}>a_{0}$, we have $\omega \in A_{\tilde{p}%
}$. Setting $a\equiv a_{0}+0,01$ we obtain $\omega \in A_{a}.$

(1$^{\circ }$) If $a\leq p^{\prime }$, then, $\omega \in A_{p^{\prime }}$
and, hence, $\omega ^{1-p}\in A_{p}.$ By Theorem \ref{Suf}, for any $u\in 
\mathbb{R}^{d}$, there holds $S_{u}:L_{p^{\prime },\omega }\hookrightarrow
L_{p^{\prime },\omega }$ and $S_{u}:L_{p,\omega ^{1-p}}\hookrightarrow
L_{p,\omega ^{1-p}}$. Now, following step by step the proof of Theorem 1.1
of \cite[p.369]{Jaw} of Jawerth, we find that $S_{u,\omega }:L_{p,\omega
}\hookrightarrow L_{p,\omega }$ for any $u\in \mathbb{R}^{d}.$

(2$^{\circ }$) If $a>p^{\prime }$, then, $\omega \in A_{a}$ and, hence, $%
\omega ^{1-a^{\prime }}\in A_{a^{\prime }}.$ By Theorem \ref{Suf}, for any $%
u\in \mathbb{R}^{d}$, there holds $S_{u}:L_{a,\omega }\hookrightarrow
L_{a,\omega }$ and $S_{u}:L_{a^{\prime },\omega ^{1-a^{\prime
}}}\hookrightarrow L_{a^{\prime },\omega ^{1-a^{\prime }}}$. Again,
following step by step the proof of Theorem 1.1 of \cite[p.369]{Jaw} of
Jawerth, we find that $S_{u,\omega }:L_{a^{\prime },\omega }\hookrightarrow
L_{a^{\prime },\omega }$ for any $u\in \mathbb{R}^{d}.$ Since $S_{u,\omega
}:L_{\infty ,\omega }\hookrightarrow L_{\infty ,\omega }$ and $S_{u,\omega
}:L_{a^{\prime },\omega }\hookrightarrow L_{a^{\prime },\omega }$, using
Marcinkiewicz interpolation theorem, for any $p\in \left( a^{\prime },\infty
\right) $ we get $S_{u,\omega }:L_{p,\omega }\hookrightarrow L_{p,\omega }$
for any $u\in \mathbb{R}^{d}.$ Namely, for $a>p^{\prime },$ we have $%
S_{u,\omega }:L_{p,\omega }\hookrightarrow L_{p,\omega }$ for any $u\in 
\mathbb{R}^{d}$, as desired.

(b) We consider the case $\omega \in A_{\infty }$, $p\in \left( 0,\infty
\right) $ and $f\in L_{p,\omega }$. This case follows from (a) and
extrapolation result Theorem \ref{ET}.
\end{proof}

\begin{proof}[\textbf{Proof of Theorem \protect\ref{Ruwf}}]
Let $0<p<\infty $, $\omega \in A_{\infty }$ and $p^{\ast }\equiv \min
\left\{ 1,p\right\} .$ Then,%
\[
\left\Vert \mathcal{R}_{u,\omega }f\right\Vert _{p,\omega }^{p^{\ast
}}=\left\Vert \sum\limits_{k=0}^{1}\tfrac{1}{2^{k}c^{k}}\left( S_{u,\omega
}\right) ^{k}f\right\Vert _{p,\omega }^{p^{\ast }}\leq
2\sum\limits_{k=0}^{1}\tfrac{1}{2^{kp^{\ast }}c^{kp^{\ast }}}\left\Vert
\left( S_{u,\omega }\right) ^{k}f\right\Vert _{p,\omega }^{p^{\ast }}\leq
4\left\Vert f\right\Vert _{p,\omega }^{p^{\ast }}.
\]
\end{proof}

\begin{proof}[\textbf{Proof of Theorem \protect\ref{UCB}}]
(a) By Remark \ref{remCc}(b), $C_{c}$ is a dense subset of $L_{p,\omega }$.
First we consider the case $0<p<1$ and prove that $F_{H}\left( u\right) $ is
bounded and uniformly continuous on $\mathbb{R}^{d}$ for functions $H\in
C_{c}$, where $q,a,r$ and $G$ is from Definition \ref{ef} with $G\in
L_{r^{\prime },\omega }$ and $\left\Vert G\right\Vert _{r^{\prime },\omega
}=1$. Boundedness of $F_{H}\left( \cdot \right) $ is easy consequence of the
H\"{o}lder's inequality (\ref{holder}) and Theorem \ref{Ruwf}. Indeed:%
\[
\left\vert F_{H}\left( u\right) \right\vert \leq \int\nolimits_{\mathbb{R}%
^{d}}\left\vert \mathcal{R}_{u,\omega }f\left( x\right) \right\vert
^{q}\left\vert G(x)\right\vert \omega \left( x\right) dx\leq \left\Vert 
\mathcal{R}_{u,\omega }f\right\Vert _{p,\omega }^{q}\left\Vert G\right\Vert
_{r^{\prime },\omega }<\infty .
\]%
On the other hand, note that $H$ is uniformly continuous on $\mathbb{R}^{d}$%
, see e.g. Lemma 23.42 of \cite[pp.557-558]{yeh} for $d=1$. Take $%
\varepsilon >0$ and $u_{1},u_{2},x\in \mathbb{R}^{d}$. Then, for this $%
\varepsilon ,$ there exists a $\delta \equiv \delta \left( \varepsilon
\right) >0$ such that%
\[
\left\vert H\left( u_{1}+x\right) -H\left( u_{2}+x\right) \right\vert \leq 
\frac{\varepsilon ^{1/q}}{2^{q}\left( 1+\left\langle \omega \right\rangle _{%
\text{supp}H}\right) }
\]%
when $\left\vert u_{1}-u_{2}\right\vert <\delta .$ Then,%
\[
\left\vert F_{H}\left( u_{1}\right) \text{-}F_{H}\left( u_{2}\right)
\right\vert \leq \int\nolimits_{\mathbb{R}^{d}}\left\vert \mathcal{R}%
_{u_{1},\omega }H\left( x\right) ^{q}\text{-}\mathcal{R}_{u_{2},\omega
}H\left( x\right) ^{q}\right\vert \left\vert G(x)\right\vert \omega \left(
x\right) dx
\]%
\[
\leq 2^{q-1}\int\nolimits_{\text{supp}H}\left\vert \mathcal{R}_{u_{1},\omega
}H\left( x\right) \text{-}\mathcal{R}_{u_{2},\omega }H\left( x\right)
\right\vert ^{q}\left\vert G(x)\right\vert \omega \left( x\right) dx
\]%
\[
\leq \frac{2^{q-1}\varepsilon }{2^{q}\left( 1+\left\langle \omega
\right\rangle _{\text{supp}H}\right) }\int\nolimits_{\text{supp}H}\left\vert
G(x)\right\vert \omega \left( x\right) dx\leq \frac{\varepsilon \left\langle
\omega \right\rangle _{\text{supp}H}}{2\left( 1+\left\langle \omega
\right\rangle _{\text{supp}H}\right) }\left\Vert G\right\Vert _{r^{\prime
},\omega }<\varepsilon .
\]

Thus conclusion of Theorem \ref{UCB} follows on $C_{c}$.

For the case $f\in L_{p,\omega }$ there exists an $H\in C_{c}$ so that%
\[
\left\Vert f-H\right\Vert _{p,\omega }<\frac{\xi ^{1/q}}{2^{1/q}\left(
1+2^{q}4^{q/p}\right) ^{1/q}}
\]%
for any $\xi >0$. Therefore%
\[
\left\vert F_{f}\left( u_{1}\right) -F_{f}\left( u_{2}\right) \right\vert
\leq \left\vert F_{f}\left( u_{1}\right) -F_{H}\left( u_{1}\right)
\right\vert +
\]%
\[
+\left\vert F_{H}\left( u_{1}\right) -F_{H}\left( u_{2}\right) \right\vert
+\left\vert F_{H}\left( u_{2}\right) -F_{f}\left( u_{2}\right) \right\vert
\]%
\[
\leq 2^{q-1}\int\nolimits_{\mathbb{R}^{d}}\left\vert \mathcal{R}%
_{u_{1},\omega }f\left( x\right) \text{-}\mathcal{R}_{u_{1},\omega }H\left(
x\right) \right\vert ^{q}\left\vert G(x)\right\vert \omega \left( x\right)
dx+\frac{\xi }{2}+
\]%
\[
+2^{q-1}\int\nolimits_{\mathbb{R}^{d}}\left\vert \mathcal{R}_{u_{2},\omega
}H\left( x\right) \text{-}\mathcal{R}_{u_{2},\omega }f\left( x\right)
\right\vert ^{q}\left\vert G(x)\right\vert \omega \left( x\right) dx
\]%
\[
\leq 2^{q-1}\left\Vert \mathcal{R}_{u_{1},\omega }\left( f-H\right)
\right\Vert _{p,\omega }^{q}+2^{q-1}\left\Vert \mathcal{R}_{u_{2},\omega
}\left( f-H\right) \right\Vert _{p,\omega }^{q}+\frac{\xi }{2}
\]%
\[
\leq 2^{q}4^{q/p}\left\Vert f-H\right\Vert _{p,\omega }^{q}+\frac{\xi }{2}%
\leq 2^{q}4^{q/p}\frac{\xi }{2\left( 1+2^{q}4^{q/p}\right) }+\frac{\xi }{2}%
\leq \frac{\xi }{2}+\frac{\xi }{2}=\xi .
\]%
As a result we have $F_{f}\in \mathcal{C}(\mathbb{R}^{d}\mathbf{)}$. In the
case $1\leq p<\infty $, proof of $F_{f}\in \mathcal{C}(\mathbb{R}^{d}\mathbf{%
)}$ is the same with minor modification of above proof.
\end{proof}

\begin{proof}[\textbf{Proof of Theorem \protect\ref{TR}}]
Let $0<p<\infty $, $\omega \in A_{\infty }$, and $0\leq f,g\in L_{p,\omega }$%
. If $\left\Vert g\right\Vert _{p,\omega }=\left\Vert f\right\Vert
_{p,\omega }$ or $\left\Vert g\right\Vert _{p,\omega }=\left\Vert
f\right\Vert _{p,\omega }=0$, then, result (\ref{concl}) is obvious. So we
assume that $\left\Vert g\right\Vert _{p,\omega },\left\Vert f\right\Vert
_{p,\omega }>0$ and $\left\Vert g\right\Vert _{p,\omega }\not=\left\Vert
f\right\Vert _{p,\omega }$. Case (1$^{\circ }$): Let $1\leq p<\infty $ and
we define, for $g\in L_{p,\omega }$, function%
\[
F_{g}\left( u,G_{0},p,\omega \right) =\int\nolimits_{\mathbb{R}^{d}}\mathcal{%
R}_{u,\omega }g\left( x\right) \left\vert G_{0}(x)\right\vert \omega \left(
x\right) dx,\quad u\in \mathbb{R}^{d}
\]%
with $G_{0}\in \mathcal{Z}\left( p,\omega \right) $ and use this to obtain%
\[
\left\Vert F_{g}\right\Vert _{\mathcal{C}\left( \mathbb{R}^{d}\right)
}=\left\Vert \int\nolimits_{\mathbb{R}^{d}}\mathcal{R}_{u,\omega }g\left(
x\right) \left\vert G_{0}(x)\right\vert \omega \left( x\right) dx\right\Vert
_{\mathcal{C}\left( \mathbb{R}^{d}\right) }
\]%
\[
\leq \sup_{u\in \mathbb{R}^{d}}\int\nolimits_{\mathbb{R}^{d}}\left\vert 
\mathcal{R}_{u,\omega }g\left( x\right) \right\vert \left\vert
G_{0}(x)\right\vert \omega \left( x\right) dx\leq \sup_{u\in \mathbb{R}%
^{d}}\left\Vert \mathcal{R}_{u,\omega }g\right\Vert _{p,\omega }\left\Vert
G_{0}\right\Vert _{p^{\prime },\omega }\leq c\left\Vert g\right\Vert
_{p,\omega }.
\]

On the other hand, for any $\varepsilon >0$ (see e.g. Theorem 18.4 of \cite%
{yeh}) we can choose appropriately an $G_{\varepsilon }\in \mathcal{Z}\left(
p,\omega \right) $ satisfying%
\[
\int\nolimits_{\mathbb{R}^{d}}f\left( x\right) \left\vert G_{\varepsilon
}\left( x\right) \right\vert \omega \left( x\right) dx\geq \left\Vert
f\right\Vert _{p,\omega }-\varepsilon ,
\]%
and one can find%
\[
\left\Vert F_{f}\left( u,G_{\varepsilon },p,\omega \right) \right\Vert _{%
\mathcal{C}\left( \mathbb{R}^{d}\right) }\geq \left\vert F_{f}\left(
0,G_{\varepsilon },p,\omega \right) \right\vert \geq \int\nolimits_{\mathbb{R%
}^{d}}\mathcal{R}_{0,\omega }f\left( x\right) \left\vert G_{\varepsilon
}(x)\right\vert \omega \left( x\right) dx
\]%
\[
\geq \int\nolimits_{\mathbb{R}^{d}}f\left( x\right) \left\vert
G_{\varepsilon }(x)\right\vert \omega \left( x\right) dx=\left\Vert
f\right\Vert _{p,\omega }-\varepsilon .
\]%
By hypothesis, we get, for any $\varepsilon >0,$%
\begin{equation}
\left\Vert f\right\Vert _{p,\omega }-\varepsilon \leq \left\Vert F_{f}\left(
\cdot ,G_{\varepsilon },p,\omega \right) \right\Vert _{\mathcal{C}\left( 
\mathbb{R}^{d}\right) }\leq c\left\Vert F_{g}\left( \cdot ,G_{0},p,\omega
\right) \right\Vert _{\mathcal{C}\left( \mathbb{R}^{d}\right) }\leq
C\left\Vert g\right\Vert _{p,\omega }.  \label{sn0}
\end{equation}%
Now taking as $\varepsilon \rightarrow 0+$ we find $\left\Vert f\right\Vert
_{p,\omega }\leq C\left\Vert g\right\Vert _{p,\omega }.$ In the general case 
$f,g\in L_{p,\omega }$ we get $\left\Vert f\right\Vert _{p,\omega }\leq
2C\left\Vert g\right\Vert _{p,\omega }.$

Case (2$^{\circ }$): Case $p\in \left( 0,1\right) $ can be obtained using
the same procedure given in the Case (1$^{\circ }$) with small
modifications. 
\[
\left\Vert F_{g}\right\Vert _{\mathcal{C}\left( \mathbb{R}^{d}\right)
}=\left\Vert \int\nolimits_{\mathbb{R}^{d}}\left( \mathcal{R}_{u,\omega
}g\left( x\right) \right) ^{q}\left\vert \tilde{G}_{0}(x)\right\vert \omega
\left( x\right) dx\right\Vert _{\mathcal{C}\left( \mathbb{R}^{d}\right) }
\]%
\[
\leq \sup_{u\in \mathbb{R}^{d}}\left\Vert \mathcal{R}_{u,\omega
}g\right\Vert _{p,\omega }^{q}\left\Vert \tilde{G}_{0}\right\Vert
_{r^{\prime },\omega }\leq c\left\Vert g\right\Vert _{p,\omega }^{q}.
\]

On the other hand, for any $\varepsilon >0$ and appropriately chosen $\tilde{%
G}_{\varepsilon }\in \mathcal{Z}\left( r,\omega \right) $ with 
\[
\int\nolimits_{\mathbb{R}^{d}}f\left( x\right) ^{q}\left\vert \tilde{G}%
_{\varepsilon }\left( x\right) \right\vert \omega \left( x\right) dx\geq
\left\Vert f\right\Vert _{p,\omega }^{q}-\varepsilon \text{,}
\]%
one can find%
\[
\left\Vert F_{f}\left( \cdot ,\tilde{G}_{\varepsilon },p,\omega \right)
\right\Vert _{\mathcal{C}\left( \mathbb{R}^{d}\right) }\geq \left\vert
F_{f}\left( 0\right) \right\vert \geq \int\nolimits_{\mathbb{R}^{d}}\left( 
\mathcal{R}_{0,\omega }f\left( x\right) \right) ^{q}\left\vert G\left(
x\right) \right\vert \omega \left( x\right) dx
\]%
\[
\geq \int\nolimits_{\mathbb{R}^{d}}f\left( x\right) ^{q}\left\vert G\left(
x\right) \right\vert \omega \left( x\right) dx\geq \left\Vert f\right\Vert
_{p,\omega }^{q}-\varepsilon .
\]%
Then by hypothesis, for any $\varepsilon >0$,%
\[
\left\Vert f\right\Vert _{p,\omega }^{q}-\varepsilon \leq \left\Vert
F_{f}\left( \cdot ,\tilde{G}_{\varepsilon },p,\omega \right) \right\Vert _{%
\mathcal{C}\left( \mathbb{R}^{d}\right) }\leq c\left\Vert F_{g}\left( \cdot ,%
\tilde{G}_{0},p,\omega \right) \right\Vert _{\mathcal{C}\left( \mathbb{R}%
^{d}\right) }\leq C\left\Vert g\right\Vert _{p,\omega }^{q}.
\]%
If we take $\varepsilon \rightarrow 0+$, then we obtain desired result $%
\left\Vert f\right\Vert _{p,\omega }\leq C\left\Vert g\right\Vert _{p,\omega
}.$ For the general case $f,g\in L_{p,\omega }$ we get $\left\Vert
f\right\Vert _{p,\omega }\leq 2C\left\Vert g\right\Vert _{p,\omega }.$
\end{proof}

\begin{proof}[\textbf{Proof of Theorem \protect\ref{Sduf}}]
Equalities in (\ref{Sduw01}) is follow from definitions:%
\[
S_{u,\omega }S_{\delta ,v}f\left( \cdot \right) =(\left\langle \omega
\right\rangle _{\left[ -1/2,1/2\right] ^{d}})^{-1}\int\limits_{\left[
-1/2,1/2\right] ^{d}}\left( S_{\delta ,v}f\right) \left( \cdot +u+t\right)
\omega \left( t\right) dt
\]%
\[
=(\left\langle \omega \right\rangle _{\left[ -1/2,1/2\right]
^{d}})^{-1}\int\limits_{\left[ -1/2,1/2\right] ^{d}}\int\limits_{\left[
-\delta /2,\delta /2\right] ^{d}}f\left( \cdot +u+t+v+s\right) ds\omega
\left( t\right) dt
\]%
\[
=\int\limits_{\left[ -\delta /2,\delta /2\right] ^{d}}(\left\langle \omega
\right\rangle _{\left[ -1/2,1/2\right] ^{d}})^{-1}\int\limits_{\left[
-1/2,1/2\right] ^{d}}f\left( \cdot +u+t+v+s\right) \omega \left( t\right)
dtds
\]%
\begin{equation}
\text{=}\int\nolimits_{\left[ -\delta /2,\delta /2\right] ^{d}}S_{u,\omega
}f\left( \cdot +v+s\right) ds\text{=}S_{\delta ,v}S_{u,\omega }f\left( \cdot
\right) .  \label{jaku}
\end{equation}%
Second equality in (\ref{Sduw01}) is follow from (\ref{jaku}). Equality in (%
\ref{InSduf}) follows from (\ref{Sduw01}) and (\ref{efef}). Now we give the
proof of inequality in (\ref{InSduf}). Since $F_{S_{\delta ,v}}$=$S_{\delta
,v}F_{f}$, we get by (\ref{sn0}) that%
\[
\left\Vert S_{\delta ,u}f\right\Vert _{p,\omega }\leq \left\Vert
F_{S_{\delta ,v}}\right\Vert _{\mathcal{C}\left( \mathbb{R}^{d}\right)
}=\left\Vert S_{\delta ,v}F_{f}\right\Vert _{\mathcal{C}\left( \mathbb{R}%
^{d}\right) }\leq \left\Vert F_{f}\right\Vert _{\mathcal{C}\left( \mathbb{R}%
^{d}\right) }\leq c\left\Vert f\right\Vert _{p,\omega }.
\]
\end{proof}

\begin{proof}[\textbf{Proof of Theorem \protect\ref{frac}}]
(1$^{\circ }$) Let $p\in \left( 1,\infty \right) $. Since $F_{V_{\delta
}f}=V_{\delta }\left( F_{f}\right) $, we get $F_{\left( V_{\delta }\right)
^{s}f}=\left( V_{\delta }\right) ^{s}\left( F_{f}\right) $ for any $s\in 
\mathbb{N}$. For any $N\in \mathbb{N}$, we have%
\[
F_{\sum\nolimits_{s=0}^{N}(-1)^{s}C_{s}^{k}\left( V_{\delta }\right)
^{s}f}=\sum\nolimits_{s=0}^{N}(-1)^{s}C_{s}^{k}\left( V_{\delta }\right)
^{s}\left( F_{f}\right) ,
\]%
\[
\left\Vert F_{\sum\nolimits_{s=0}^{N}(-1)^{s}C_{s}^{k}\left( V_{\delta
}\right) ^{s}f}\right\Vert _{\mathcal{C}\left( \mathbb{R}^{d}\right)
}=\left\Vert \sum\nolimits_{s=0}^{N}(-1)^{s}C_{s}^{k}\left( V_{\delta
}\right) ^{s}\left( F_{f}\right) \right\Vert _{\mathcal{C}\left( \mathbb{R}%
^{d}\right) }
\]%
\[
\leq \sum\nolimits_{s=0}^{N}\left\vert C_{s}^{k}\right\vert \left\Vert
F_{f}\right\Vert _{\mathcal{C}\left( \mathbb{R}^{d}\right) }\leq \left\Vert
F_{f}\right\Vert _{\mathcal{C}\left( \mathbb{R}^{d}\right)
}+\sum\nolimits_{s=1}^{\infty }\frac{c\left( k\right) }{s^{1+k}}\left\Vert
F_{f}\right\Vert _{\mathcal{C}\left( \mathbb{R}^{d}\right) }
\]%
\[
<c\left\Vert F_{f}\right\Vert _{\mathcal{C}\left( \mathbb{R}^{d}\right) }
\]%
by $\left\vert C_{s}^{k}\right\vert \leq c_{k}s^{-1-k}$ (\cite[p.14, (1.51)]%
{skm}). Now using Corollary \ref{Coroper1} and Theorem \ref{TR} we obtain%
\[
\left\Vert \left( E-V_{\delta }\right) ^{k}f\right\Vert _{p,\omega
}=\lim_{N\rightarrow \infty }\left\Vert
\sum\nolimits_{s=0}^{N}(-1)^{s}C_{s}^{k}\left( V_{\delta }\right)
^{s}f\right\Vert _{p,\omega }
\]%
\[
\leq \lim_{N\rightarrow \infty }\left\Vert
F_{\sum\nolimits_{s=0}^{N}(-1)^{s}C_{s}^{k}\left( V_{\delta }\right)
^{s}f}\right\Vert _{\mathcal{C}\left( \mathbb{R}^{d}\right)
}=\lim_{N\rightarrow \infty }\left\Vert
\sum\nolimits_{s=0}^{N}(-1)^{s}C_{s}^{k}\left( V_{\delta }\right)
^{s}\left( F_{f}\right) \right\Vert _{\mathcal{C}\left( \mathbb{R}%
^{d}\right) }
\]%
\[
\leq 2^{k}\left\Vert F_{f}\right\Vert _{\mathcal{C}\left( \mathbb{R}%
^{d}\right) }\leq c\left\Vert f\right\Vert _{p,\omega }.
\]

(2$^{\circ }$) For general case $p\in \left( 0,\infty \right) $, we use (1$%
^{\circ }$) and Theorem \ref{ET} to finish proof.
\end{proof}

\begin{proof}[\textbf{Proof of Theorem \protect\ref{TRver}}]
Let $1<p<\infty $, $\omega \in A_{\infty }$, and $\left( f,g_{j}\right) \in 
\mathcal{F}$ with $f,g_{j}\in L_{p,\omega }$. If $\left\Vert
g_{j}\right\Vert _{l_{s}^{m}\left( L_{p,\omega }\right) }=\left\Vert
f\right\Vert _{p,\omega }=0$, then, result (\ref{hkm}) is obvious. So we
assume that $\left\Vert g_{j}\right\Vert _{l_{s}^{m}\left( L_{p,\omega
}\right) },\left\Vert f\right\Vert _{p,\omega }>0$. Since $\omega \in
A_{\infty }$, there is $a_{0}\equiv e^{2^{11+d}\left[ \omega \right]
_{\infty }}>1$ (\cite[p.786]{hyPrz13}) such that, we obtain $\omega \in
A_{a} $ with $a\equiv a_{0}+0,01$. Let $\Psi $ be $1$-finite family of open
bounded cubes $Q_{i}$ of $\mathbb{R}^{d}$ having Lebesgue measure $1$, such
that $\left( \cup _{i}Q_{i}\right) \cup A=\mathbb{R}^{d}$ for some null-set $%
A$. Then,%
\[
\left\Vert F_{g_{j}}\right\Vert _{L_{a}}^{a}=\int\nolimits_{\mathbb{R}%
^{d}}\left\vert F_{g_{j}}\left( u\right) \right\vert ^{a}du=\sum_{Q_{i}\in
\Psi }\int\nolimits_{Q_{i}}\left\vert F_{g_{j}}\left( u\right) \right\vert
^{a}du
\]%
\[
\leq \sum_{Q_{i}\in \Psi }\int\nolimits_{Q_{i}}c^{a}\left\Vert
g_{j}\right\Vert _{p,\omega }^{a}\chi _{Q_{i}}\left( u\right)
du=c^{a}\left\Vert g_{j}\right\Vert _{p,\omega }^{a}\sum_{Q_{i}\in \Psi
}\int\nolimits_{Q_{i}}\chi _{Q_{i}}\left( u\right) du=c^{a}\left\Vert
g_{j}\right\Vert _{p,\omega }^{a}.
\]%
In this case%
\[
\left\Vert F_{f}\right\Vert _{L_{a}}\leq c\left\Vert F_{g_{j}}\right\Vert
_{l_{s}^{m}\left( L_{a}\right) }=c\left( \sum\nolimits_{j=0}^{m}\left\Vert
F_{g_{j}}\right\Vert _{L_{a}}^{s}\right) ^{1/s}
\]%
\[
=c\left( \sum\nolimits_{j=0}^{m}\left\Vert g_{j}\right\Vert _{p,\omega
}^{s}\right) ^{1/s}=c\left\Vert g\right\Vert _{l_{s}^{m}\left( L_{p,\omega
}\right) }.
\]%
On the other hand,%
\[
\left\Vert F_{f}\right\Vert _{L_{a}}=\left( \int\nolimits_{\mathbb{R}%
^{d}}\left\vert F_{f}\left( u\right) \right\vert ^{a}du\right) ^{1/a}\geq
\left( \int\nolimits_{\left[ 0,1\right] ^{d}}\left\vert F_{f}\left( u\right)
\right\vert ^{a}du\right) ^{1/a}\geq \left\Vert f\right\Vert _{p,\omega }.
\]%
Combining these inequalities we get%
\begin{equation}
\left\Vert f\right\Vert _{p,\omega }\leq \left\Vert F_{f}\right\Vert
_{L_{a}}\leq c\left\Vert F_{g_{j}}\right\Vert _{l_{s}^{m}\left( L_{a}\right)
}\leq c\left\Vert g\right\Vert _{l_{s}^{m}\left( L_{p,\omega }\right) }.
\label{sn0a}
\end{equation}
\end{proof}

\begin{proof}[\textbf{Proof of Theorem \protect\ref{ShrpMarch}}]
Let $r\in \mathbb{N}$, $\omega \in A_{\infty }$, $p\in \left( 1,\infty
\right) $, $\delta \in \left( 0,\infty \right) $, and $f\in L_{p,\omega }$.
Then there is $a_{0}\equiv e^{2^{11+d}\left[ \omega \right] _{\infty }}>1$ (%
\cite[p.786]{hyPrz13}) such that, we obtain $\omega \in A_{a}$ with $a\equiv
a_{0}+0,01$. Then there exist $m\in \mathbb{N}$ such that%
\[
\left\Vert \left( E\text{-}T_{\delta }\right) ^{2r}\left( F_{f}\right)
\right\Vert _{L_{a}}^{s}\geq c\sum\limits_{j=0}^{m}2^{-j2rs}K_{\Delta
^{r}}\left( F_{f},\left( 2^{j}\delta \right) ^{2r+2},L_{a}\right) ^{s},
\]%
for $s\equiv \max \left\{ a,2\right\} $. On the other hand, we know that%
\[
\left\Vert \left( E-V_{\delta }\right) ^{r}\left( F_{f}\right) \right\Vert
_{L_{a}}\approx K_{\Delta ^{r}}\left( F_{f},\delta ^{2r},L_{a}\right)
\approx \left\Vert \left( E-T_{\delta }\right) ^{2r}\left( F_{f}\right)
\right\Vert _{L_{a}}.
\]%
As a consequence,%
\[
\left\Vert \left( E\text{-}V_{\delta }\right) ^{r}\left( F_{f}\right)
\right\Vert _{L_{a}}^{s}\geq c\sum\limits_{j=0}^{m}2^{-j2rs}\left\Vert
\left( E\text{-}V_{\delta }\right) ^{r+1}\left( F_{f}\right) \right\Vert
_{L_{a}}^{s}.
\]%
From the last inequality and Theorem \ref{TRver} we obtain (\ref{InShrpMarch}%
).
\end{proof}

\begin{proof}[Proof of Theorem \protect\ref{Jacks}]
It is enough to proof%
\begin{equation}
A_{2\sigma }\left( f\right) _{p\left( \cdot \right) }\leq c\left\Vert \left(
I-V_{1/\left( 2\sigma \right) }\right) ^{r}f\right\Vert _{p\left( \cdot
\right) }.  \label{JE}
\end{equation}%
Let $g_{\sigma }$ be an exponential type entire function of degree $\leq
\sigma $, belonging to $\mathcal{C}(\mathbb{R}^{d})$, as the best
approximation of $F_{f}\in \mathcal{C}(\mathbb{R}^{d})$. Since $F_{J\left(
f,\sigma \right) }=J\left( F_{f},\sigma \right) $ and $J\left( g_{\sigma
},\sigma \right) =g_{\sigma }$, there holds%
\[
A_{2\sigma }\left( f\right) _{p,\omega }\leq \left\Vert f-J\left( f,\sigma
\right) \right\Vert _{p,\omega }\leq c\left\Vert F_{f-J\left( f,\sigma
\right) }\right\Vert _{\mathcal{C}(\mathbb{R}^{d})}=c\left\Vert
F_{f}-F_{J\left( f,\sigma \right) }\right\Vert _{\mathcal{C}(\mathbb{R}^{d})}
\]%
\[
=c\left\Vert F_{f}-J\left( F_{f},\sigma \right) \right\Vert _{\mathcal{C}(%
\mathbb{R}^{d})}=c\left\Vert F_{f}-g_{\sigma }+g_{\sigma }-J\left(
F_{f},\sigma \right) \right\Vert _{\mathcal{C}(\mathbb{R}^{d})}
\]%
\[
=c\left\Vert F_{f}-g_{\sigma }+J\left( g_{\sigma },\sigma \right) -J\left(
F_{f},\sigma \right) \right\Vert _{\mathcal{C}(\mathbb{R}^{d})}=c\left\Vert
F_{f}-g_{\sigma }+J\left( g_{\sigma }-F_{f},\sigma \right) \right\Vert _{%
\mathcal{C}(\mathbb{R}^{d})}
\]%
\[
\leq c(A_{\sigma }\left( F_{f}\right) _{\mathcal{C}(\mathbb{R}%
^{d})}+cA_{\sigma }\left( F_{f}\right) _{\mathcal{C}(\mathbb{R}%
^{d})})=cA_{\sigma }\left( F_{f}\right) _{\mathcal{C}(\mathbb{R}^{d})}.
\]

Therefore%
\[
A_{2\sigma }\left( f\right) _{p,\omega }\leq cA_{\sigma }\left( F_{f}\right)
_{\mathcal{C}(\mathbb{R}^{d})}\leq c\left\Vert \left( I-T_{\frac{1}{2\sigma }%
}\right) ^{2r}\left( F_{f}\right) \right\Vert _{\mathcal{C}(\mathbb{R}%
^{d})}\leq c\left\Vert \left( I-V_{\frac{1}{2\sigma }}\right) ^{r}\left(
F_{f}\right) \right\Vert _{\mathcal{C}(\mathbb{R}^{d})}
\]%
\[
=c\left\Vert F_{\left( I-V_{1/\left( 2\sigma \right) }\right)
^{r}f}\right\Vert _{\mathcal{C}(\mathbb{R}^{d})}\leq c\left\Vert \left(
I-V_{1/\left( 2\sigma \right) }\right) ^{r}f\right\Vert _{p,\omega }.
\]
\end{proof}

\end{document}